\documentclass[a4paper]{amsart}

\usepackage[a4paper, margin=2cm]{geometry}
\usepackage[british]{babel}

\usepackage{amsmath}
\usepackage{amssymb}
\usepackage{calc}
\usepackage{enumitem}
\usepackage{graphicx}
\usepackage{url}
\usepackage{xcolor}
\usepackage{mathtools}
\usepackage{hyperref}

\hypersetup{
  colorlinks   = true, 
  urlcolor     = {blue!50!black}, 
  linkcolor    = {blue!50!black}, 
  citecolor   = {red!50!black} 
}

\numberwithin{equation}{section}

\theoremstyle{plain}
\newtheorem{thm}{Theorem}[section]

\newtheorem{prop}[thm]{Proposition}
\newtheorem{lem}[thm]{Lemma}

\newtheorem*{lem*}{Lemma}
\newtheorem*{cor*}{Corollary}

\theoremstyle{definition}
\newtheorem{defn}[thm]{Definition}

\newtheorem{notn}[thm]{Notation}
\newtheorem{qn}[thm]{Question}
\newtheorem*{qn*}{Question}

\newtheorem*{term*}{Terminology}

\theoremstyle{remark}
\newtheorem{rem}[thm]{Remark}

\DeclareMathOperator{\Age}{Age}
\DeclareMathOperator{\Ao}{Age_\omega}

\DeclareMathOperator{\Aut}{Aut}

\DeclareMathOperator{\dom}{dom}

\DeclareMathOperator{\id}{id}

\DeclareMathOperator{\qftp}{qftp}

\newcommand{\mb}[1]{\mathbb{#1}}
\newcommand{\mc}[1]{\mathcal{#1}}
\newcommand{\ov}[1]{\overline{#1}}

\newcommand{\sub}{\subseteq}

\newcommand{\tld}[1]{\widetilde{#1}}

\newcommand{\fin}{\subseteq_{\text{fin\!}}}

\newcommand{\ra}{\rightarrow}

\newcommand{\la}{\leftarrow}

\newcommand{\h}[1]{\widehat{#1}}
\newcommand{\emp}{\varnothing}

\newcommand{\Fr}{Fra\"{i}ss\'{e} }

\newcommand{\Frthm}{Fra\"{i}ss\'{e}'s theorem }
\newcommand{\Ka}{Kat\v{e}tov }

\allowdisplaybreaks

\makeatletter
\def\author@andify{%
  \nxandlist {\unskip ,\penalty-1 \space\ignorespaces}%
    {\unskip {} \@@and~}%
    {\unskip \penalty-2 \space \@@and~}%
}
\makeatother

\thanks{The first author is funded by Project 24-12591M of the Czech Science Foundation (GA\v{C}R)}

\author{Rob Sullivan}
\address{\parbox{\linewidth}{Rob Sullivan\\
Institute of Computer Science, Czech Academy of Sciences\\
Pod Vodárenskou věží 271/2\\
182 00 Prague\\
Czech Republic
}}
\email{robertsullivan1990+maths@gmail.com}

\author{Jeroen Winkel}
\address{Jeroen Winkel}
\email{winkeljeroen+maths@gmail.com}



\title{An unusual example of a universal automorphism group}

\date{\today}

\begin{document}

\begin{abstract}
    Let $M$ be a Fra\"{i}ss\'{e} structure (a countably infinite ultrahomogeneous structure). We refer to the class of structures embeddable in $M$ as the $\omega$-age of $M$. We consider the following two properties of $M$: we say that $M$ has a \emph{universal automorphism group} if, for each $A$ in the $\omega$-age of $M$, there is an embedding $\textrm{Aut}(A) \to \textrm{Aut}(M)$, and we say that $M$ has \emph{group-extensible $\omega$-age} if, for each $A$ in the $\omega$-age of $M$, there is an embedding $A \to M$ such that each automorphism of the image extends to an automorphism of $M$ and the extension map preserves group composition. It is immediate that if $M$ has group-extensible $\omega$-age, then $M$ has a universal automorphism group. We give an example of a Fra{\"{i}}ss{\'{e}} structure with a universal automorphism group whose $\omega$-age is not group-extensible, showing that the above two properties are not equivalent. 
\end{abstract}

\maketitle

\section{Introduction}

\subsection{Universal automorphism groups}

A first-order structure $M$ is \emph{ultrahomogeneous} if each isomorphism between finitely generated substructures of $M$ extends to an automorphism of $M$. We assume the reader is familiar with ultrahomogeneous structures: see \cite{Mac11}, \cite{Hod93} for background. In this paper, we are concerned with the \emph{universality} of automorphism groups of such structures. Given a countable ultrahomogeneous structure $M$, we define the \emph{$\omega$-age} of $M$, written $\Ao(M)$, to be the class of structures embeddable in $M$ (including infinite structures). We say that $\Aut(M)$ is \emph{universal} if, for each $A \in \Ao(M)$, there is a group embedding $\Aut(A) \to \Aut(M)$. Note that, a priori, the group embedding need not have any relation to the structure.

We call a countably infinite ultrahomogeneous structure $M$ a \emph{\Fr structure}. Many \Fr structures are known to have universal automorphism groups. M{\"{u}}ller proved in \cite{Mul16} that $\Aut(M)$ is universal for any \Fr structure $M$ with a local stationary independence relation: examples of such structures are those with free amalgamation, the generic poset and the rational Urysohn space. In addition, \cite[Theorem 4.4]{KSW25} gives that any linearly ordered \Fr structure with a local stationary weak independence relation has a universal automorphism group: examples include the generic ordered graph, the generic order expansion of the rational Urysohn space and the generic $n$-linear order (the \Fr limit of the class of finite structures $A = (\dom(A), <_0, \cdots, <_{n-1})$ where each $<_i$ is a linear order on $\dom(A)$). Other examples of \Fr structures with universal automorphism groups are the countable atomless Boolean algebra, the generic lattice, the generic semilattice and the random tournament (\cite{KM17}). See \cite{Hen71}, \cite{Kat88}, \cite{Usp90}, \cite{MW92}, \cite{Bil12}, \cite{BM13} for earlier results; there is a discussion of these in the introduction of \cite{KSW25}.

\subsection{Group-extensive embeddings} In each example in the literature known to the authors, the method for showing universality of $\Aut(M)$ is similar: we call it a \emph{tower construction}. We first define some terminology:
\begin{defn}
    Let $M$ be a countable ultrahomogeneous structure. An embedding $f : A \to M$ is \emph{group-extensive} if there exists a group embedding $\eta : \Aut(f(A)) \to \Aut(M)$ with $g \sub \eta(g)$ for all $g \in \Aut(f(A))$. If each $A \in \Ao(M)$ admits a group-extensive embedding $A \to M$, we say that $\Ao(M)$ is \emph{group-extensible}.
\end{defn}

For each $A \in \Ao(M)$, the tower construction builds a group-extensive embedding $A \to M$. Universality of $\Aut(M)$ then immediately results from the group-extensibility of $\Ao(M)$. The most general formulation of the tower construction is due to Kubi\'{s} and Ma\v{s}ulovi\'{c}, who provided a categorical version in \cite{KM17} via what they termed \emph{\Ka functors}.

We provide a simple example to illustrate the idea of how group-extensive embeddings are built using the tower construction, giving a sketch for the Rado graph, denoted by $\Gamma$. Given a countable graph $A$, inductively construct $A = M_0 \sub M_1 \sub \cdots$ as follows: to construct $M_k$ from $M_{k - 1}$, for each $F \fin M_{k - 1}$, we add a new vertex $v_F \in M_k$ adjacent to exactly $F$, and we do not add any edges between the new vertices. Note that, for $k \geq 1$, each $g \in \Aut(M_{k-1})$ extends uniquely to an element $\eta_k(g)$ of $\Aut(M_k)$ via $v_F \mapsto v_{g(F)}$. Also note that $\eta_k : g \mapsto \eta_k(g)$ is a group embedding $\Aut(M_{k-1}) \to \Aut(M_k)$. Let $M_\omega = \bigcup_{k < \omega} M_k$ and $\eta = \bigcup_{k < \omega} \eta_k$; then $\eta$ is a group embedding $\Aut(A) \to \Aut(M_\omega)$ with $\eta(g) \supseteq g$ for all $g \in \Aut(A)$. To show that there is an isomorphism $\iota : M_\omega \to \Gamma$, one checks that for each pair of finite disjoint sets $U, V \sub M_\omega$, there is some ``witness" $w \in M_\omega \setminus (U \cup V)$ with $w$ adjacent to all of $U$ and none of $V$: as $U, V \sub M_{k-1}$ for some $k$, the new vertex of $M_k$ whose $M_{k-1}$-neighbourhood is $U$ then gives the required witness $w$. The map $\iota \circ \eta : \Aut(A) \to \Aut(\Gamma)$ then gives the required group-extensive embedding.

\subsection{Failure of universality} It was unknown for a long time whether there was in fact any \Fr structure with a non-universal automorphism group. (Jaligot asked this in \cite{Jal07}.) The first such examples of non-universality were given in \cite{KS20} by Kubi{\'{s}} and Shelah. They showed that universality can fail dramatically, constructing for instance examples of \Fr structures $M$ in a finite relational language where $\Aut(M)$ is torsion-free and the class $\{\Aut(A) \mid A \in \Age(M)\}$ contains all finite symmetric groups. \cite[Theorem 5.1]{KSW25} gives more examples, showing that some well-known \Fr structures have non-universal automorphism groups: for example, the semigeneric tournament and the generic $n$-anticlique-free oriented graph.

\subsection{New results}

In \cite{KSW25}, the authors and A.\ Kwiatkowska considered a range of \Fr structures $M$, showing either the universality of $\Aut(M)$ via group-extensive embeddings (built with the tower construction, using the framework of \Ka functors from \cite{KM17}), or the non-universality of $\Aut(M)$ via ad hoc methods particular to the structure in question.

In discussions around \cite{KSW25}, Kwiatkowska, Barto{\v{s}}, Kubi{\'{s}} and Vaccaro asked the following:

\begin{qn} \label{q: Ola qn}
    Let $M$ be a \Fr structure. If $\Aut(M)$ is universal, does this imply that $\Ao(M)$ is group-extensible?
\end{qn}

This question resulted from the observation that, for all previous examples in the literature (those known to the authors, at least), universality of $\Aut(M)$ was shown by proving the group-extensibility of $\Ao(M)$. Another way of viewing this question is that it asks whether the existence of an abstract group embedding $\Aut(A) \to \Aut(M)$ must always be ``structurally witnessed" by an embedding $A \to M$.

The main result of this short paper gives a negative answer to Question \ref{q: Ola qn}.

\begin{thm} \label{t: main thm}
    There exists a \Fr structure $\mb{D}$ in a finite relational language, where $\mb{D}$ is transitive and has strong amalgamation, such that $\Aut(\mb{D})$ is universal and $\Ao(\mb{D})$ is not group-extensible.
\end{thm}

(In the above theorem we use the following standard terminology: a countable ultrahomogeneous structure $M$ is \emph{transitive} if the permutation action $\Aut(M) \curvearrowright M$ is transitive, and $M$ has \emph{strong amalgamation} if, for each pair of embeddings $B \la A \ra C$ in $\Age(M)$, there is an amalgam $B \ra D \la C$ in $\Age(M)$ where the intersection of the images of $B$ and $C$ in $D$ is exactly the image of $A$.)

We show universality of $\Aut(\mb{D})$ via a new method: for each $A \in \Ao(\mb{D})$, rather than using the tower construction as in previous examples, we produce a group embedding $\Aut(A) \to \Aut(\mb{D})$ by constructing an $L$-structure $M \cong \mb{D}$ on which $\Aut(A)$ acts faithfully by automorphisms of $M$. We conjecture that this new method should be applicable more generally, showing the universality of $\Aut(M)$ for other \Fr structures $M$, though in this paper we only concern ourselves with the particular structure $\mb{D}$.

\begin{rem}
    \cite[Example Two, pg.\ 13]{HS19} can be adapted to give an example of a \emph{finite} ultrahomogeneous structure $B$ such that $\Aut(B)$ is universal and $\Age(B)$ is not group-extensible, as follows. Let $L_R = \{R\}$ be a language consisting of a binary relation symbol $R$, and let $B$ be the $L_R$-structure with domain $\{a, a', b_0, b_1, b_2, b_3\}$ and \[R^B = \{(a, b_0), (a, b_2), (a', b_1), (a', b_3), (a, a), (a', a')\} \cup \{(b_i, b_{i+1}) \mid i < 4\},\] where addition is (mod $4$). Note that $B$ is a digraph with loops. It is straightforward to check that $B$ is ultrahomogeneous and $\Aut(B)$ is a cyclic group of order $4$ with generator $f \in \Aut(B)$ specified by \[f(a) = a', f(a') = a, f(b_i) = b_{i+1} \text{ for } i < 4,\] and that $\Aut(B)$ is universal. Let $A$ be the substructure induced by $B$ on $\{a, a'\}$, and let $g \in \Aut(A)$ be the involution swapping $a, a'$. It is straightforward to see that $\id_A, g$ are the only embeddings of $A$ in $B$, and any $h \in \Aut(B)$ extending $g$ has order $4$. So $\Age(B)$ is not group-extensible. Note that $B$ is not transitive and does not have strong amalgamation. We thank W.\ Kubi\'{s} for pointing this example out to us.
\end{rem}

\section{Notation and terminology} \label{s: action notn and term}

We write $\lambda : G \curvearrowright V$ for an action of a group $G$ on a set $V$. For brevity, we will often say that $V$ is a \emph{$G$-set} to mean that $V$ is equipped with an action $\lambda : G \curvearrowright V$. We write $gv$ to mean $g \cdot v$, and we extend this notation to sets in the natural manner: for $U \sub V$, we write $gU = \{gu \mid u \in U\}$. For $v \in V$ we also write $Gv = \{gv \mid g \in G\}$. Let $S$ be a structure with domain $V$. If the action $\lambda : G \curvearrowright V$ is by automorphisms of $S$, we will say that $S$ is a \emph{$G$-structure}. When we refer to a $G$-structure with domain $V$, where $V$ is a $G$-set, we always assume the actions on the $G$-structure and on the $G$-set are the same action. We also use this terminology with the specific name of a structure: for example, a $G$-oriented graph is an oriented graph equipped with an action of $G$ by automorphisms.
    
Let $V$ be a $G$-set. For $U \sub V$, we write $G_{(U)}$ for the pointwise-stabiliser of $U$ and $G_{\{U\}}$ for the setwise-stabiliser of $U$. We say that $V$ is \emph{transitive} if it is a single orbit, and \emph{faithful} if the action $G \curvearrowright V$ is faithful.

Given a set $V$, we write $[V]^n$ for the collection of subsets of $V$ of size $n$, each of which we call an \emph{$n$-set}. Given an action $\lambda : G \curvearrowright V$, we often consider the element-wise action of $G$ \emph{induced} by $\lambda$ on $[V]^n$: the action $G \curvearrowright [V]^n$ induced by $\lambda$ is defined by $g \cdot U = \{gu \mid u \in U\}$, for $U \in [V]^n$.
    
Given a structure $A$, we let $\rho_A$ denote the natural permutation action $\Aut(A) \curvearrowright A$.

\section{The structure}

\begin{defn} \label{str def}
    Let $L_R = \{R\}$ be a language consisting of a binary relation symbol $R$. We call an $L_R$-structure $A$ an \emph{$I_3$-free oriented graph} if $R^A$ is an oriented graph relation (irreflexive and asymmetric) and the substructure on any $3$-set contains at least one oriented edge.

    Let $L_S$ be a language consisting of a $4$-ary relation symbol $S$. We say that an $L_S$-structure $A$ is a \emph{semifinal structure} if $S^A = \bigcup_{B \sub A, |B| = 4} S^B$ and, for each $B \sub A$, $|B| = 4$, there is an enumeration $b_0, b_1, b_2, b_3$ of the vertices of $B$ such that \[S^B = \{(b_0, b_1, b_2, b_3), (b_2, b_3, b_0, b_1)\}.\] We refer to $S^B$ as the \emph{semifinal on $B$}. (The name \emph{semifinal structure} is intended to evoke the semifinals of a football tournament where $b_0$ wins against $b_1$ and $b_2$ wins against $b_3$.)

    Let $L = L_R \cup L_S$. We let $\mc{D}$ denote the class of finite $L$-structures $A$ such that $(\dom(A), R^A)$ is an $I_3$-free oriented graph and $(\dom(A), S^A)$ is a semifinal structure. It is straightforward to see that the class $\mc{D}$ has strong amalgamation: given $B, C \in \mc{D}$ with $B \cap C = A$, define an amalgam of $B$, $C$ over $A$ by adding oriented edges $(b, c)$ for all $b \in B \setminus A$, $c \in C \setminus A$ and by defining $S$ arbitrarily on any set of four vertices which is not contained in $B$ or in $C$. We denote the \Fr limit of $\mc{D}$ by $\mb{D}$, and we let $\ov{\mc{D}} = \Ao(\mb{D})$.
\end{defn}

We first give conditions for a $G$-set to admit a $G$-structure in $\ov{\mc{D}}$; we use these in Lemmas \ref{l: extend I_3-free to 4-tuples of A} and \ref{l: ind sf}.

\begin{lem} \label{l: cond for I_3-free}
    Let $G$ be a group, and $V$ a $G$-set. The following are equivalent:
    \begin{enumerate}[label=(\roman*)]
        \item \label{i: no swap cond} for all $U \in [V]^3$, there are distinct $u, v \in U$ such that there is no $g \in G$ swapping $u, v$;
        \item \label{i: I_3 free extra cond} there is an $I_3$-free $G$-oriented graph $A$ with domain $V$ such that, for all distinct $u, v \in V$, if there is no edge between $u, v$, then there is $g \in G$ swapping $u, v$;
        \item \label{i: I_3 free cond} there is an $I_3$-free $G$-oriented graph $A$ with domain $V$. 
    \end{enumerate}
\end{lem}
\begin{proof}
    \ref{i: no swap cond} $\Rightarrow$ \ref{i: I_3 free extra cond}: Define an oriented graph $A = (V, R^A)$ as follows: for each orbit $O \sub [V]^2$ of the element-wise action $G \curvearrowright [V]^2$ induced by $G \curvearrowright V$, take an orbit representative $\{u, v\}$, and if there is no $g \in G$ swapping $u, v$, orient $\{u, v\}$ as $(u, v)$ and add $G(u, v)$ to $R^A$. As there is no $g \in G$ swapping $u, v$, this gives a well-defined orientation on each pair of vertices in $O$. If there is $g \in G$ swapping $u, v$, do not add an edge to $R^A$ for any pair of vertices in $O$. It is immediate that $(V, R^A)$ is a $G$-oriented graph such that, for all distinct $u, v \in V$, if there is no edge between $u, v$, then there is $g \in G$ swapping $u, v$, and \ref{i: no swap cond} immediately gives that $(V, R^A)$ is $I_3$-free.

    \ref{i: I_3 free extra cond} $\Rightarrow$ \ref{i: I_3 free cond}: trivial.
    
    \ref{i: I_3 free cond} $\Rightarrow$ \ref{i: no swap cond}: As $A$ is $I_3$-free, for any $U \sub A$, $|U| = 3$, there are $u, v \in U$ with $R(u, v)$, and thus there is no $g \in G$ swapping $u, v$, as $G$ acts by automorphisms of $A$.
\end{proof}

\begin{lem} \label{l: double tr from sf}
    Let $G$ be a group, and $V$ a $G$-set. The following are equivalent:
    \begin{enumerate}[label=(\roman*)]
        \item \label{i: double tr cond} for each $U \in [V]^4$, there is an enumeration $v_0, v_1, v_2, v_3$ of $U$ such that the action of each element of $G_{\{U\}} \setminus G_{(U)}$ on $U$ is the double transposition $(v_0 v_2)(v_1 v_3)$;
        \item \label{i: sf on dom} there is a $G$-semifinal structure $A$ with domain $V$.
    \end{enumerate}
\end{lem}
\begin{proof}
    \ref{i: double tr cond} $\Rightarrow$ \ref{i: sf on dom}: For each orbit $O$ of $G \curvearrowright [V]^4$, we enumerate an orbit representative of $O$ as $v_0, v_1, v_2, v_3$ so that condition \ref{i: double tr cond} of Lemma \ref{l: double tr from sf} is satisfied, and add $G(v_0, v_1, v_2, v_3) \cup G(v_2, v_3, v_0, v_1)$ to $S^A$. It is straightforward to verify that this gives a $G$-semifinal structure $A = (V, S^A)$.

    \ref{i: sf on dom} $\Rightarrow$ \ref{i: double tr cond}: Let $U \in [V]^4$, and let $v_0, v_1, v_2, v_3$ be an enumeration of $U$ with $(v_0, v_1, v_2, v_3) \in S^A$; condition \ref{i: double tr cond} for $U$ immediately follows from the fact that $A$ is a $G$-semifinal structure, so any $g \in G_{\{U\}} \setminus G_{(U)}$ preserves $S^A$.
\end{proof}

We use the \ref{i: sf on dom} $\Rightarrow$ \ref{i: double tr cond} direction of the above Lemma \ref{l: double tr from sf} frequently throughout the paper -- as we use it so often and its proof is straightforward, we do not explicitly reference it, and simply state, for example, ``this follows from the semifinal on $U$".

Note that Lemmas \ref{l: cond for I_3-free} and \ref{l: double tr from sf} imply that, given a countable $G$-set $V$, there is a $G$-structure $A \in \ov{\mc{D}}$ with domain $V$ if and only if condition \ref{i: no swap cond} of Lemma \ref{l: cond for I_3-free} and condition \ref{i: double tr cond} of Lemma \ref{l: double tr from sf} both hold.

\begin{lem} \label{l: groups with D str}
    Let $G$ be a countable group, and consider $G$ as a $G$-set via the left-multiplicative action. Then there is a $G$-structure $A \in \ov{\mc{D}}$ with domain $G$ if and only if $G$ does not have a subgroup isomorphic to $C_2 \times C_2$ or $C_4$.
\end{lem}
\begin{proof}
    $\Rightarrow$: Let $H \leq G$, $|H| = 4$. By Lemma \ref{l: double tr from sf}, the $G$-set $G$ satisfies condition \ref{i: double tr cond} of the lemma, and so as the setwise-stabiliser of $H$ contains $H$ itself, we have $H \not\cong C_2 \times C_2$ and $H \not\cong C_4$.

    $\Leftarrow$: As observed immediately before the statement of the current lemma, it suffices to show that condition \ref{i: no swap cond} of Lemma \ref{l: cond for I_3-free} and condition \ref{i: double tr cond} of Lemma \ref{l: double tr from sf} hold. We first show condition \ref{i: no swap cond} of Lemma \ref{l: cond for I_3-free}. Let $g, h \in G \setminus \{1\}$, $g \neq h$. Consider the $3$-set $\{1, g, h\}$: if there are elements of $G$ swapping each pair of elements, then we have $g^2 = 1$, $h^2 = 1$ and $gh = hg$, so $\{1, g, h, gh\} \cong C_2 \times C_2$, contradiction. So condition \ref{i: no swap cond} of Lemma \ref{l: cond for I_3-free} holds for $\{1, g, h\}$, and thus by translation for each $3$-set. Let $U \in [G]^4$. For $g \in G$, if $g$ fixes a point of $U$ then $g = 1$, so each element of $G_{\{U\}} \setminus G_{(U)}$ acts as a $4$-cycle or a double transposition on $U$: as $G$ does not contain a subgroup isomorphic to $C_4$ or $C_2 \times C_2$, each element of $G_{\{U\}} \setminus G_{(U)}$ must act on $U$ as the same double transposition, so we have condition \ref{i: double tr cond} of Lemma \ref{l: double tr from sf}.
\end{proof}

\section{The main result}

\subsection{Failure of group-extensibility}
\begin{prop} \label{p: omega-age of D not group-extensible}
    $\ov{\mc{D}}$ is not group-extensible.
\end{prop}
\begin{proof}
    This is similar to the proof of \cite[Lemma 3.21]{KSW25}. Let $A \in \mc{D}$ with domain $\{a, b, c\}$ and oriented edges $(a, b), (a, c)$ (note that as $|A| = 3$, there is no semifinal on $A$). Let $f : A \to \mb{D}$ be an embedding. By the extension property, there is $v \in \mb{D}$ with oriented edges $(v, f(b)), (f(c), v)$ and no edge between $v$ and $f(a)$. (There is also necessarily a semifinal on $\{v, f(a), f(b), f(c)\}$, but it will not play a role in the argument.) Suppose for a contradiction that $f$ is group-extensive. Then the involution of $f(A)$ swapping $f(b), f(c)$ and fixing $f(a)$ extends to an involution $\tau \in \Aut(\mb{D})$. As $v$ is oriented oppositely to $f(b)$ and $f(c)$, we have $\tau(v) \neq v$. But then $\{f(a), v, \tau(v)\}$ has no oriented edge, contradiction.
\end{proof}

\begin{rem}
    Let $M$ be the generic $I_3$-free oriented graph (the \Fr limit of the class of finite $I_3$-free oriented graphs). The same argument as in the preceding proof shows that $\Ao(M)$ is not group-extensible. In fact $\Aut(M)$ is not universal: in the proof of \cite[Proposition 3.23]{KSW25}, it is shown that $\Aut(M)$ does not contain a pair of commuting involutions, and as there is an element of $\Age(M)$ with automorphism group $C_2 \times C_2$, this gives the non-universality of $\Aut(M)$. In the case of the structure $\mb{D}$ of the present paper, the proof of \cite[Proposition 3.23]{KSW25} also shows that $\Aut(\mb{D})$ does not contain a pair of commuting involutions. But it is not difficult to see from Lemma \ref{l: double tr from sf} (which relies on the semifinal structure) that, for each $A \in \ov{\mc{D}}$, the group $\Aut(A)$ does not have a subgroup $\cong C_2 \times C_2$. We add the semifinal relation $S$ specifically to prevent failure of universality.
\end{rem}

\subsection{Universality of \texorpdfstring{$\Aut(\mb{D})$}{Aut(D)}} \hfill

\subsubsection{Overview of the construction.} Let $A \in \ov{\mc{D}}$ and let $G = \Aut(A)$. We produce a group embedding $G \to \Aut(\mb{D})$ by constructing a faithful $G$-structure $M \cong \mb{D}$. We build $M$ inductively as the union of a chain $M_0 \sub M_1 \sub \cdots$: given $M_{i-1} \in \ov{\mc{D}}$, a faithful $G$-structure $M_{i-1}$ and a one-point extension $C \to E$, where $C \fin M_{i-1}$, $E \in \mc{D}$, we extend $M_{i-1}$ (and the action) to a faithful $G$-structure $M_i \in \ov{\mc{D}}$ containing a realisation of $C \to E$. Via a standard bookkeeping argument, we ensure for all $i < \omega$ that each one-point extension over a finite substructure of $M_i$ has a realisation in some $M_j$, $j > i$, which shows that $M = \bigcup_{i < \omega} M_i$ has the extension property for $\mc{D}$; thus $M \cong \mb{D}$.

The \textbf{key idea} of the construction is the observation that there are faithful actions $G \curvearrowright V$ where, for each $v \in V$, the point-stabiliser of $v$ is equal to the pointwise-stabiliser of a finite subset of $A$ of size $\geq 4$ in the permutation action $\rho_A$. We say that such actions are \emph{$A$-nice}. For example, the coordinate-wise action $G \curvearrowright (A)^n$ induced by $\rho_A$ is $A$-nice for each $n \geq 4$. We build the $M_i$ using this: we define $L$-structures $M_i$ on $A$-nice $G$-sets $V$, ensuring that $M_i \in \ov{\mc{D}}$ and that the action is via automorphisms. The fact that each point-stabiliser in $V$ is equal to the pointwise-stabiliser of some $A' \fin A$, $|A'| = 4$, gives us ``enough information from the structure $A$" to build the $L$-structure $M_i$ on the $G$-set $V$ (bearing in mind that the semifinal relation $S$ is $4$-ary): using the fact that the oriented graph relation $R^A$ of $A$ is $I_3$-free and $(\dom(A), S^A)$ is a semifinal structure, we construct $M_i$ such that $(\dom(M_i), R^{M_i})$ is an $I_3$-free oriented graph and $(\dom(M_i), S^{M_i})$ is a semifinal structure.

\subsubsection{Building structure on nice \texorpdfstring{$G$}{G}-sets}

For this section and the next section, we fix $A \in \ov{\mc{D}}$ with $|A| \geq 4$ and let $G = \Aut(A)$.

\begin{defn} \label{d: nice}
    We say that a $G$-set $V$ is \emph{$A$-nice} if the action $G \curvearrowright V$ is faithful and, for each $v \in V$, there is $A' \fin A$, $|A'| \geq 4$, with $G_v = G_{(A')}$.

    Let $V$ be an $A$-nice $G$-set. For $v \in V$, we define $\h{v} = \{a \in A \mid ga = a \text{ for all } g \in G_v\}$. As $V$ is $A$-nice, for each $v \in V$ we have $|\h{v}| \geq 4$ and $G_v = G_{(\h{v})}$. It is also straightforward to see that for $g \in G$, $v \in V$, we have $\h{gv} = g\h{v}$.
\end{defn}

\begin{lem} \label{l: extend I_3-free to 4-tuples of A}
    Let $V$ be an $A$-nice $G$-set. Then there exists an $I_3$-free $G$-oriented graph relation $(V, R)$ satisfying the condition that, for all distinct $u, v \in V$, if there is no edge between $u, v$ then there is $g \in G$ swapping $u, v$.
\end{lem}
\begin{proof}
    By Lemma \ref{l: cond for I_3-free}, it suffices to show condition \ref{i: no swap cond} in the statement of that lemma. Suppose for a contradiction that there is $U = \{u_0, u_1, u_2\} \in [V]^3$ and $g_0, g_1, g_2 \in G$ with $g_i$ swapping $u_i, u_{i+1}$ for all $i < 3$. (Here addition is (mod $3$).) For each $i$, as $g_i \notin G_{u_i}$ and $G_{u_i} = G_{(\h{u_i})}$, there is $a_i \in \h{u_i}$ such that $g_i a_i \neq a_i$ and $g_i$ swaps $a_i, g_i a_i$. If some $g_i$ were to have more than one fixed point in $A$, then there would be a $4$-set on which $g_i$ swapped two points and fixed two points, contradicting the fact that there is a semifinal on this $4$-set and $g_i$ acts as an automorphism. So each $g_i$ has at most one fixed point in $A$, and thus as $|\h{u_0}| \geq 4$, there is $a \in \h{u_0}$ with $|\{a, g_0a, g_1g_0a\}| = 3$. As $g_2g_1g_0 \in G_{u_0} = G_{(\h{u_0})}$, we have $g_2g_1g_0a = a$, and so as all transpositions of $\{a, g_0a, g_1g_0a\}$ occur by the actions of $g_0, g_1, g_2$, the oriented graph structure induced by $A$ on this $3$-set is an anticlique: contradiction.
\end{proof}

\begin{lem} \label{l: sf on 4-set}
    Let $V$ be an $A$-nice $G$-set, and let $U \in [V]^4$. Then there is an enumeration $v_0, v_1, v_2, v_3$ of $U$ such that the action of each element of $G_{\{U\}} \setminus G_{(U)}$ on $U$ is the double transposition $(v_0 v_2)(v_1 v_3)$. 
\end{lem}
\begin{proof}
    Let $g \in G_{\{U\}} \setminus G_{(U)}$. Take distinct $t, u \in U$ with $gt = u$.
    
    First consider the case where there are distinct $v, w \in U \setminus \{t, u\}$ such that $g$ acts on $U$ as $(t u v)(w)$. Take $a \in \h{t}$ with $ga \neq a$. As $g^3t = t$ and $V$ is $A$-nice, we have $g^3a = a$. If $g^2a = a$ then $ga = a$, contradiction. So $a, ga, g^2a$ are distinct, and $g$ setwise-stabilises $\{a, ga, g^2a\}$. As $gw = w$, for each $b \in \h{w}$ we have $gb = b$, and so a fortiori there is $b \in \h{w}$ with $b \notin \{a, ga, g^2a\}$. So $g$ acts on the $4$-set $\{a, ga, g^2a, b\} \sub A$ as $(a\, ga\, g^2a)(b)$, giving a contradiction as $A$ has a semifinal structure.

    Now consider the case where there are distinct $v, w \in U \setminus \{t, u\}$ such that $g$ acts on $U$ as $(t u v w)$. As $t \neq v$, there is $a \in \h{t}$ with $g^2a \neq a$. Then $a, ga, g^2a$ are distinct. If $g^3a \in \{a, ga, g^2a\}$ we obtain a contradiction as $g^4a = a$, and if $g^3a \not\in \{a, ga, g^2a\}$ then $g$ acts on $\{a, ga, g^2a, g^3a\}$ as a $4$-cycle, contradicting the semifinal structure on $A$.

    So $gu = t$. Let $a \in \h{t}$ with $a \neq ga$; as $g^2t = t$ we have $g^2a = a$. Let $v, w \in U \setminus \{t, u\}$ be distinct. As $|\h{v}| \geq 4$, there are $b, c \in \h{v}$ such that $\{a, ga, b, c\}$ is a $4$-set. Thus $g$ cannot fix each of $b, c$, as $A$ has a semifinal structure. So $g$ swaps $v, w$, and hence $g$ acts on $U$ as the double transposition $(t u)(v w)$. By the same argument as for $g$, each element of $G_{\{U\}} \setminus G_{(U)}$ acts as a double transposition on $U$, and it remains to check that each $g' \in G_{\{U\}} \setminus G_{(U)}$ acts as $(tu)(vw)$.

    Let $g' \in G_{\{U\}} \setminus G_{(U)}$. Suppose for a contradiction that $g'$ does not act on $U$ as $(tu)(vw)$. By relabelling $v, w$ if necessary, we may assume that $g'$ acts on $U$ as $(tv)(uw)$. We thus have $[g, g'] \in G_t$ (where $[g, g']$ is the commutator), and so $gg'z = g'gz$ for all $z \in \h{t}$. Take $b \in \h{t}$ with $g'b \neq b$, and recall that in the previous paragraph we took $a \in \h{t}$ with $ga \neq a$. If $g$ fixes each of $b, g'b$, then this contradicts the semifinal on the $4$-set $\{a, ga, b, g'b\}$. If $g$ fixes $b$ and moves $g'b$, then $g'b = g'gb = gg'b$, contradiction, and likewise we obtain a contradiction if $g$ fixes $g'b$ and moves $b$. It remains to consider the case where $g$ moves each of $b$, $g'b$. If $gb \neq g'b$, then $\{b, gb, g'b, gg'b\}$ is a $4$-set, and as $g$ acts on this $4$-set as $(b \, gb)(g'b \, gg'b)$ and the $4$-set has a semifinal, we have that $g'$ also acts as $(b \, gb)(g'b \, gg'b)$, so $g'gb = b$, contradiction. So the only remaining case is where $g$ moves each of $b$, $g'b$ and $gb = g'b$.

    As $g'gt = w$ and $w \neq t$, there is $c \in \h{t}$ with $g'gc \neq c$. By an analogous argument to that in the previous paragraph, where we replace $b$, $g'b$ with $c$, $g'gc$, we may assume that $g$ moves each of $c$, $g'gc$ and $gc = g'gc$, as all other cases already give a contradiction. We have $g'c = g'g^2c = g^2c = c$ and $g'(g'gc) = gc = g'gc$, so $g'$ fixes each of $c$, $g'gc$ and swaps $b, g'b$. But this contradicts the semifinal on $\{b, g'b, c, g'gc\}$. We have thus obtained a contradiction in all cases, so $g'$ acts on $U$ as $(tu)(vw)$ as required.
\end{proof}

\begin{lem} \label{l: ind sf}
    Let $V$ be an $A$-nice $G$-set. Let $P \sub [V]^4$ be a union of orbits of the element-wise action $G \curvearrowright [V]^4$. Let $\tld{S} \sub V^4$ be a relation such that:
    \begin{itemize}
        \item for all $v_0, v_1, v_2, v_3 \in V$, if $\tld{S}(v_0, v_1, v_2, v_3)$ then $\{v_0, v_1, v_2, v_3\} \in P$;
        \item for all $U \in P$, we have that $(U, \tld{S}|_U)$ is a semifinal structure;
        \item $(V, \tld{S})$ is a $G$-structure.
    \end{itemize} 
    Then there exists a semifinal $G$-structure $(V, S)$ with $\tld{S} \sub S$.
\end{lem}
\begin{proof}
    We first observe that by Lemma \ref{l: sf on 4-set}, for all $U \in [V]^4 \setminus P$, there is an enumeration $v_0, v_1, v_2, v_3$ of $U$ such that the action of each element of $G_{\{U\}} \setminus G_{(U)}$ on $U$ is the double transposition $(v_0 v_2)(v_1 v_3)$. The remainder of the proof is essentially that of the direction \ref{i: double tr cond} $\Rightarrow$ \ref{i: sf on dom} of Lemma \ref{l: double tr from sf}. Define a semifinal structure $(V, S)$ with $\tld{S} \sub S$ as follows: for each orbit $O \sub [V]^4 \setminus P$, take the enumeration $(v_0, v_1, v_2, v_3)$ given by Lemma \ref{l: sf on 4-set} of an orbit representative of $O$, and add $G(v_0, v_1, v_2, v_3) \cup G(v_2, v_3, v_0, v_1)$ to $S$. By the assumptions of the current lemma, this gives a semifinal $G$-structure $(V, S)$.
\end{proof}

\subsubsection{The inductive step}

Recall that we have fixed $A \in \ov{\mc{D}}$ with $|A| \geq 4$ and let $G = \Aut(A)$.

We now show how to carry out the inductive step of the construction, building an $A$-nice $G$-structure witnessing a one-point extension.

\begin{lem} \label{l: exist nice}
    Let $A' \fin A$, $|A'| \geq 4$. Then there exists an $A$-nice transitive $G$-set $V$ such that there is $u \in V$ with $G_u = G_{(A')}$.
\end{lem}
\begin{proof}
    Let $W$ be the subset of $A^{|A'|}$ consisting of the tuples where all elements are distinct, and equip $W$ with the element-wise action induced by the permutation action $\rho_A : G \curvearrowright A$. It is straightforward to see that the $G$-set $W$ is $A$-nice. Let $\bar{a}'$ be an enumeration of $A'$. Then we have $\bar{a}' \in W$, and $G_{\bar{a}'} = G_{(A')}$. Let $V = G\bar{a}'$: then $V$ is as required.
\end{proof}

\begin{notn}
    We write $\qftp(b/A)$ for the quantifier-free type of $b$ over the parameter set $A$.
\end{notn}

\begin{lem} \label{l: extend by 1-type}
    Let $B_0$ be an $A$-nice $G$-structure. Let $C_0 \fin B_0$, and let $E' = C_0 \cup \{e'\} \in \mc{D}$ with $e' \notin C_0$. Then there exists an $A$-nice $G$-structure $B$ with $B_0 \sub B$ such that there is $e \in B$ with $\qftp(e/C_0) = \qftp(e'/C_0)$.
\end{lem}
\begin{proof}
    By Lemma \ref{l: exist nice}, there is an $A$-nice transitive $G$-set $N$ such that there is $e \in N$ with $G_e = G_{(C_0)}$. By taking a $G$-isomorphic copy of $N$ if necessary, we may assume $N \cap B_0 = \emp$. Let $V = \dom(B_0) \cup N$. Let $\lambda : G \curvearrowright V$ be the action extending the actions $G \curvearrowright B_0$ and $G \curvearrowright N$. Let $R_0$ be the oriented graph relation of $B_0$. We extend $R_0$ to an oriented graph relation $R \sub V^2$ as follows. Define that $e$ has the same quantifier-free type over $C_0$ as $e'$ in the $L_R$-reduct of $E'$ (that is, $\qftp_{(V, R)}(e/C_0) = \qftp_{E'|_{L_R}}(e'/C_0)$), and define $R(e, v)$ for all $v \in B_0 \setminus C_0$. For each $g \in G$, $v \in B_0$, define $\qftp_{(V, R)}(ge, gv) = \qftp_{(V, R)}(e, v)$; note that this is consistent, as $G_e = G_{(C_0)}$ by how we defined $e$. Finally, we define $R$ on $N$ by taking the oriented graph relation given by Lemma \ref{l: extend I_3-free to 4-tuples of A} applied to $N$. It is immediate that $\lambda$ acts by automorphisms of $(V, R)$, and as $\lambda|_{B_0}$ is faithful, we have that $\lambda$ is faithful.
    
    We now check that $R$ is $I_3$-free. For triples of vertices in $N$, this follows by Lemma \ref{l: extend I_3-free to 4-tuples of A}, and for triples where one vertex is in $N$ and the other two vertices are in $B_0$ this follows immediately by the definition of $R$. We now check triples where two vertices are in $N$ and the other is in $B_0$: as $\lambda$ acts by automorphisms of $(V, R)$, it suffices to consider the case of triples $\{e, u, v\}$ with $u \in N \setminus \{e\}$ and $v \in B_0$. If there is no $g \in \Aut(A)$ swapping $e, u$, then as $R$ was constructed using Lemma \ref{l: extend I_3-free to 4-tuples of A} on $N$, by the conditions of Lemma \ref{l: extend I_3-free to 4-tuples of A} we have that there is an edge between $e, u$. So it remains to consider the case where there is $g \in \Aut(A)$ swapping $e, u$. If $v \notin C_0$, then $R(e, v)$ by definition. If $g^{-1}v \notin C_0$, then $R(ge, v)$, as $R(e, g^{-1}v)$ and $\lambda$ acts by automorphisms of $(V, R)$. In the case $\{v, g^{-1}v\} \sub C_0$: as $g$ swaps $e, u$ and $G_e = G_{(C_0)}$ we have $g^2 \in G_{(C_0)}$. So if $v \neq g^{-1}v$, there is no edge between $v, g^{-1}v$, and so as $E' \in \mc{D}$ there is an edge between $e, v$ or an edge between $e, g^{-1}v$; if the latter holds then we have an edge between $ge = u$ and $v$. The only remaining subcase is where $v = g^{-1}v$. As $e \neq ge$, there is some $a \in \h{e}$ with $a \neq ga$. As $|\h{v}| \geq 4$, there are $a', a'' \in \h{v}$ such that $\{a, ga, a', a''\}$ is a $4$-set, and as $g$ swaps $a, ga$ and fixes $a', a''$ this contradicts the semifinal structure on $A$. So $v$ cannot be fixed by $g$; we have now checked all cases of triples.

    Let $S_0$ be the semifinal relation of $B_0$. Define $S' \sub V^4$ extending $S_0$ by $\qftp_{(V, S')}(e, C_0) = \qftp_{E'|_{L_S}}(e'/C_0)$. For each $U = \{u\} \cup U_0 \in [V]^4$ with $u \in N$ and $U_0 \sub B_0$, as $N$ is a $G$-orbit and $N \cap B_0 = \varnothing$, we have that $u$ is fixed by each element of $G_{\{U\}}$, and thus by Lemma \ref{l: sf on 4-set} we have $G_{\{U\}} = G_{(U)}$. For each such $U$, considering the $G$-orbit of $U$, we define $\tld{S}$ extending $S'$ as follows: if some $4$-set $W$ in the $G$-orbit of $U$ is such that $S' \cap W^4 \neq \varnothing$, add $\bigcup_{g \in G} g(S' \cap W^4)$ to $\tld{S}$, and otherwise give $U$ an arbitrary semifinal structure $S'_U$ and add $\bigcup_{g \in G} gS'_U$ to $\tld{S}$. For each $U = \{u\} \cup U_0$ we then have that $(U, \tld{S}|_U)$ is a semifinal structure, as $G_{\{U\}} = G_{(U)}$, and we have that $\lambda$ acts by automorphisms of $(V, \tld{S})$. Let \[P = [\dom(B_0)]^4 \cup \{U \in [V]^4 \mid |U \cap N| = 1, |U \cap \dom(B_0)| = 3\}.\] Apply Lemma \ref{l: ind sf} with $V, \lambda, P, \tld{S}$ to obtain a semifinal structure $(V, S)$ with $\tld{S} \sub S$ such that $\lambda$ acts by automorphisms of $(V, S)$. Let $B = (V, R, S)$. Then the $G$-structure $B$ is as required.
\end{proof}

\subsubsection{Finishing the construction}

\begin{prop} \label{p: Aut(D) univ}
    The structure $\mb{D}$ has a universal automorphism group.
\end{prop}
\begin{proof}
    Let $A \in \ov{\mc{D}}$. We show that $\Aut(A)$ embeds in $\Aut(\mb{D})$.
    
    First consider the case where $|A| \geq 4$. Let $W$ be an $A$-nice $G$-set given by Lemma \ref{l: exist nice} (we may take any $A' \fin A$, $|A'| \geq 4$). Let $M_0$ be the $L$-structure with domain $W$ and with $R^{M_0}$, $S^{M_0}$ given by Lemma \ref{l: extend I_3-free to 4-tuples of A} and Lemma \ref{l: ind sf}. Then $M_0$ is an $A$-nice $G$-structure and $M_0 \in \ov{\mc{D}}$. By a standard bookkeeping argument as in the proof of \Frthm (see for example \cite{Hod93}), using Lemma \ref{l: extend by 1-type} we construct an increasing chain $M_0 \sub M_1 \sub \cdots$ of $A$-nice $G$-structures such that the union $M := \bigcup_{i < \omega} M_i$ has the extension property for $\mc{D}$ (including over the empty structure). We then have $M \cong \mb{D}$. We consider $M$ as a $G$-structure by taking the union of the $G$-actions on each $M_i$. As each $M_i$ is an $A$-nice $G$-structure, we have that $G \curvearrowright M$ is faithful, and so we obtain an embedding $\Aut(A) \to \Aut(M)$. As $M \cong \mb{D}$, there is an embedding $\Aut(A) \to \Aut(\mb{D})$ as required.

    It remains to consider the case $|A| < 4$. For $A \in \mc{D}$, $|A| < 4$, we have $\Aut(A) \cong 1$, $C_2$ or $C_3$ (an easy check), and so it suffices to show that $C_6$ embeds in $\Aut(\mb{D})$. Let $G = C_6$. By Lemma \ref{l: groups with D str}, there is a faithful $G$-structure $A \in \mc{D}$ with domain $G$. As $|A| = 6$, by the prior paragraph there is an embedding $\Aut(A) \to \Aut(\mb{D})$, and as $A$ is a faithful $G$-structure, there is an embedding $G \to \Aut(A)$; composition gives an embedding $G \to \Aut(\mb{D})$.
\end{proof}

\begin{rem}
    Lemma \ref{l: groups with D str} and Proposition \ref{p: Aut(D) univ} show that each countable group not containing a subgroup isomorphic to $C_2 \times C_2$ or $C_4$ embeds into $\Aut(\mb{D})$.
\end{rem}

\begin{proof}[Proof of Theorem \ref{t: main thm}]
    Immediate by Proposition \ref{p: omega-age of D not group-extensible} and Proposition \ref{p: Aut(D) univ}.
\end{proof}


\bibliographystyle{alpha}
\bibliography{references}

\end{document}